\documentclass[11pt]{amsart}
\usepackage{amsfonts,amssymb,amscd,amsmath,enumerate,verbatim}
\usepackage[latin1]{inputenc}
\usepackage{amscd}
\usepackage{latexsym}
\usepackage{pstcol,pst-plot,pst-3d}
\usepackage{mathptmx}
\usepackage{multicol}

\psset{unit=0.7cm,linewidth=0.8pt,arrowsize=2.5pt 4}

\newpsstyle{fatline}{linewidth=1.5pt}
\newpsstyle{fyp}{fillstyle=solid,fillcolor=verylight}
\definecolor{verylight}{gray}{0.97}
\definecolor{light}{gray}{0.9}
\definecolor{medium}{gray}{0.85}

\def\NZQ{\mathbb}               
\def\NN{{\NZQ N}}

\def\frk{\mathfrak}               

\def\Phi{{\frk N}}
\def\opn#1#2{\def#1{\operatorname{#2}}} 
\opn\chara{char} \opn\length{\ell} \opn\pd{pd} \opn\rk{rk}
\opn\projdim{proj\,dim} \opn\injdim{inj\,dim} \opn\rank{rank}
\opn\depth{depth} \opn\grade{grade} \opn\height{height}
\opn\embdim{emb\,dim} \opn\codim{codim}

\opn\Tr{Tr} \opn\bigrank{big\,rank}
\opn\superheight{superheight}\opn\lcm{lcm}
\opn\trdeg{tr\,deg}
\opn\reg{reg} \opn\lreg{lreg} \opn\ini{in} \opn\lpd{lpd}
\opn\size{size}\opn{\mult}{mult}
\opn\div{div} \opn\Div{Div} \opn\cl{cl} \opn\Cl{Cl}
\opn\Spec{Spec} \opn\Supp{Supp} \opn\supp{supp} \opn\Sing{Sing}
\opn\Ass{Ass} \opn\Min{Min}
\opn\Ann{Ann} \opn\Rad{Rad} \opn\Soc{Soc}
\opn\Syz{Syz} \opn\Im{Im} \opn\Ker{Ker} \opn\Coker{Coker}
\opn\Am{Am} \opn\Hom{Hom} \opn\Tor{Tor} \opn\Ext{Ext}
\opn\End{End} \opn\Aut{Aut} \opn\id{id} \opn\ini{in}

\opn\nat{nat}
\opn\pff{pf}
\opn\Pf{Pf} \opn\GL{GL} \opn\SL{SL} \opn\mod{mod} \opn\ord{ord}
\opn\Gin{Gin}
\opn\Hilb{Hilb}\opn\adeg{adeg}\opn\std{std}\opn\ip{infpt}
\opn\Pol{Pol}
\opn\sat{sat}
\opn\Var{Var}
\opn\Gen{Gen}

\opn\aff{aff} \opn\con{conv} \opn\relint{relint} \opn\st{st}
\opn\lk{lk} \opn\cn{cn} \opn\core{core} \opn\vol{vol}
\opn\link{link} \opn\star{star}
\opn\gr{gr}

\def\pot#1#2{#1[\kern-0.28ex[#2]\kern-0.28ex]}

\opn\dirlim{\underrightarrow{\lim}}
\opn\inivlim{\underleftarrow{\lim}}
\let\iso=\cong

\let\to=\rightarrow

\def\Implies{\ifmmode\Longrightarrow \else
        \unskip${}\Longrightarrow{}$\ignorespaces\fi}
\def\implies{\ifmmode\Rightarrow \else
        \unskip${}\Rightarrow{}$\ignorespaces\fi}
\def\iff{\ifmmode\Longleftrightarrow \else
        \unskip${}\Longleftrightarrow{}$\ignorespaces\fi}

\let\:=\colon
\newtheorem{Theorem}{Theorem}[section]

\newtheorem{Conjecture}[Theorem]{Conjecture}

\let\epsilon\varepsilon
\let\phi=\varphi
\let\kappa=\varkappa
\def\qed{\ifhmode\textqed\fi
      \ifmmode\ifinner\quad\qedsymbol\else\dispqed\fi\fi}
\def\textqed{\unskip\nobreak\penalty50
       \hskip2em\hbox{}\nobreak\hfil\qedsymbol
       \parfillskip=0pt \finalhyphendemerits=0}
\def\dispqed{\rlap{\qquad\qedsymbol}}

\opn\dis{dis}
\def\pnt{{\raise0.5mm\hbox{\large\bf.}}}

\opn\Lex{Lex}

\newcommand{\inD}[1][\relax]{\def\argone{#1}\def\temprelax{\relax}
  \ifx\argone\temprelax\right.\else\,\middle|#1\right.{}\fi}

\newif\ifbinary
\binarytrue

\begin{document}

\title{Finite lattices and Gr\"obner bases}

\author{J\"urgen Herzog and  Takayuki Hibi}
\subjclass{}

\address{J\"urgen Herzog, Fachbereich Mathematik, Universit\"at Duisburg-Essen, Campus Essen, 45117
Essen, Germany} \email{juergen.herzog@uni-essen.de}

\address{Takayuki Hibi, Department of Pure and Applied Mathematics, Graduate School of Information Science and Technology,
Osaka University, Toyonaka, Osaka 560-0043, Japan}
\email{hibi@math.sci.osaka-u.ac.jp}

\thanks{}

\begin{abstract}
Gr\"obner bases of binomial ideals arising from finite lattices will be studied.  In terms of Gr\"obner bases and initial ideals, a characterization of finite distributive lattices
as well as planar distributive lattices will be given.
\end{abstract}
\subjclass{13P10, 06A11, 06B05}
\maketitle

\section*{Introduction}
Let $L$ be a finite lattice and $K[L]$ the polynomial ring in $|L|$ variables over a field $K$ whose variables are the elements of $L$.
A binomial of $K[L]$  of the form $ab - (a\wedge b)(a\vee b)$ is called a basic binomial.  Let $I_L$ denote the ideal of $K[L]$ generated by basic binomials
of $K[L]$.  The ideal $I_L$ was first introduced by \cite{H}.  It is shown in \cite{H} that $I_L$ is a prime ideal if and only if $L$ is a distributive
lattice, see also \cite[Theorem 10.1.3]{HH}.   When $L$ is distributive, the set of basic binomials of $K[L]$ is a Gr\"obner basis with respect to any rank reverse lexicographic order.

In the present paper, by studying Gr\"obner bases of $I_L$, a Gr\"obner basis characterization of distributive lattices as well as
planer distributive lattices will be obtained.  Moreover, we discuss the problem when $I_L$ has a quadratic Gr\"obner basis
with respect to any monomial order.

\section{Characterization of distributive lattices}
We refer the reader to \cite{St} for fundamental materials on finite lattices.
Let $\hat{0}$ (resp.\ $\hat{1}$) denote a unique minimal (resp.\ maximal)
element of a finite lattice.
Recall that a finite lattice $L$ is called {\em modular} if $x \leq b$ implies
$x\vee(a\wedge b) = (x \vee a)\wedge b$ for all $x,a,b \in L$.
A finite lattice is modular if and only if no sublattice of $L$ is isomorphic
the pentagon lattice of Figure $1$.  A finite lattice is called {\em distributive}
if, for all $x,y,z \in L$, the distributive laws
$x\wedge(y\vee z) = (x\wedge y)\vee(y\wedge z)$
and
$x\vee(y\wedge z) = (x\vee y)\wedge(y\vee z)$
hold.
Every modular lattice is distributive.
A modular lattice is distributive if and only if
no sublattice of $L$ is isomorphic to the diamond lattice of Figure $1$.

\begin{figure}[h]
\begin{center}
\psset{unit=0.5cm}
\begin{pspicture}(-10.3,-2.5)(4,3.5)

\psline(2,3)(0.5,1.5)
\psline(0.5,1.5)(0.5,0)
\psline(0.5,0)(2,-1)
\psline(2,-1)(3.5,0.75)
\psline(3.5,0.75)(2,3)

\rput(2,3){$\bullet$}
\put(1.9,3.3){$a$}
\rput(0.5,1.5){$\bullet$}
\put(-0.1,1.5){$b$}
\rput(0.5,0){$\bullet$}
\put(-0.1,0){$c$}
\rput(2,-1){$\bullet$}
\put(1.9,-1.6){$e$}
\rput(3.5,0.75){$\bullet$}
\put(3.8,0.75){$d$}

\psline(-9,3)(-11,1)
\psline(-9,3)(-9,1)
\psline(-9,3)(-7,1)
\psline(-11,1)(-9,-1)
\psline(-9,1)(-9,-1)
\psline(-7,1)(-9,-1)

\rput(-9,3){$\bullet$}
\put(-9.1,3.3){$a$}
\rput(-11,1){$\bullet$}
\put(-11.6,0.9){$b$}
\rput(-9,1){$\bullet$}
\put(-8.6,0.9){$c$}
\rput(-7,1){$\bullet$}
\put(-6.7,1){$d$}
\rput(-9,-1){$\bullet$}
\put(-9.1,-1.6){$e$}
\rput(-8.7,-2.5){Diamond Lattice}
\rput(1.7,-2.5){Pentagon Lattice}
\end{pspicture}
\end{center}
\caption{}\label{Fig1}
\end{figure}

A finite lattice is called {\em pure} if all maximal chains
between $\hat{0}$ and $\hat{1}$ have the same length.
When a finite lattice is pure, then the {\em rank function}
of $L$ can be defined.  More precisely, if $L$ is a finite pure lattice
and $a \in L$, then the {\em rank} of $a$ in L, denoted by
$\rank(a)$, is the largest integer $r$ for which there exists a chain of $L$ 
of the form
\[
\hat{0} = a_0 < a_1 < \cdots < a_r = a.
\]
If a finite lattice $L$ is modular, then one has the equality
\begin{eqnarray}
\label{modular}
\rank(p) + \rank(q) = \rank(p \wedge q) + \rank(p \vee q)
\end{eqnarray}
for all $p, q \in L$.

\medskip
Let $K$ be  field and $L$ a finite lattice. We consider the polynomial ring  $K[L]$ whose variables correspond to the elements of $L$, and define the  ideal $I_L\subset K[L]$  as the binomial ideal whose  generators are   all  basic binomials attached to $L$. As explained in the introduction,   a binomial of the form $ab-cd$ with  $c=a \vee b$ and $c=a\wedge b$ is called a {\em basic binomial}. 
The residue class ring $K[L]/I_L$ will be denoted by $R(L)$

We refer the reader to \cite{HH} for basic terminologies and notation on Gr\"obner bases. 
A {\em  rank reverse lexicographic order} on $K[L]$ is the reverse lexicographic order with the property that $a>b$ if $\rank(a)>\rank(b)$.

\begin{Theorem}
\label{easy}
Let $L$ be a finite modular  lattice. Then  the following conditions are equivalent:
\begin{enumerate}
\item[{\em (i)}] $L$ is a distributive lattice;
\item[{\em (ii)}] $I_L$ has a squarefree Gr\"obner basis
with respect to any rank reverse lexicographic order.
\end{enumerate}
\end{Theorem}

\begin{proof}
The implication (i) $\Rightarrow$ (ii) is well known, see \cite{H} and  \cite[Theorem 10.1.3]{HH}.

(ii)  $\Rightarrow$ (i):
Suppose that $L$ is a finite modular lattice which is not distributive.
Then $L$ contains the diamond lattice of Figure \ref{Fig1}.
Since $L$ is modular, one has the equality (\ref{modular})
for all $p, q \in L$.  Hence if $g = pq - p'q'$ is a basic monomial
of $I_L$, then
\[
\rank(p) + \rank(q) = \rank(p') + \rank(q').
\]
In particular the ranks of $b,c$ and $d$ coincide.
We fix a rank reverse lexicographic order $<$ with the property that
$d < q$ for all $q \in L$ with $\rank(q) = \rank(d)$.
Our work is to show that $\ini_<(I_L)$ cannot be squarefree.
Suppose, on the contrary, that $\ini_<(I_L)$ is squarefree.

First we claim
$ad^2e\in \ini_<(I_L)$. In fact, $ad^2e-a^2e^2\in I_L$, because
\[
ad^2e-a^2e^2=d\bigl(d(ae-bc)+c(bd-ae)\bigr)+ae(cd-ae).
\]
Since $\ini_<(I_L)$ is squarefree and since $ad^2e\in \ini_<(I_L)$,
it follows that $ade \in \ini_<(I_L)$.
Hence
there exists a binomial $f = ade - u \in I$,
where $u$ is a monomial of degree $3$, with
$\ini_<(f) = ade$.

Let $f = \sum_{i=1}^{N} x_if_i$, where each $x_i$ is a variable
and where each $f_i = v_i - w_i$ is a basic binomial of $I_L$,
such that $x_1v_1 = ade$ and $x_iw_i = x_{i+1}v_{i+1}$ for all
$1 \leq i < N$.
A crucial fact is that each variable appearing in $x_if_i$
belongs to the interval $[e,a]$ of $L$.
To see why this is true, we observe that
if $f_i = v_i - w_i$ is a basic binomial of $I_L$
and if each variable appearing in $v_i$ belongs to $[e,a]$,
then each variable appearing in $w_i$ must belong to $[e,a]$.
Now, since $x_1v_1 = ade$ and $x_iw_i = x_{i+1}v_{i+1}$
for all $1 \leq i < N$,
this observation guarantees that each variable appearing in $x_if_i$
belongs to the interval $[e,a]$ of $L$.
In particular $u = x_Nw_N$ consists of variables
belonging to $[e,a]$, say $u = \ell m n$.

Now, one has $f = ade - \ell m n \in I$,
where $\ell, m$ and $n$ belong to $[e,a]$.
Let $\ell \geq m \geq n$.  Since we are working
with a rank reverse lexicographic order, it follows that
$e$ is the smallest variable among all variables belonging to $[e,a]$.
Since $\ini_<(f) = ade$, one has $n = e$.

On the other hand, since $I_L$ is generated by basic binomials of $L$,
it follows easily that
if $g = p_1p_2\cdots p_r - q_1q_2 \cdots q_r$ is a binomial
belonging to $I_L$, then
\[
\sum_{i=1}^{r} \rank(p_i) = \sum_{i=1}^{r} \rank(q_i).
\]
Thus in particular one has
\[
\rank(a) + \rank(d) = \rank(\ell) + \rank(m).
\]
Since $a$ is a unique maximal element of $[a,e]$, it follows that
$\rank(a) \geq \rank(\ell) (\geq \rank(m))$.
Hence $\rank(d) \leq \rank(m)$. If $\rank(d) =\rank(m)$, then $d<m$ by the given order of the variables. On the other hand, if $\rank(d) < \rank(m)$, then $d<m$, since we use a rank reverse lexicographic order. Thus in any case $d<m$, and this implies that $\ini_<(f) = \ell m n$, a contradiction.
Consequently, the monomial $ade$ cannot belong to $\ini_<(I_L)$.
Hence $ad^2e$ belongs to a unique minimal set of monomial generators
of $\ini_<(I_L)$.  Thus $\ini_<(I_L)$ cannot be squarefree.
\end{proof}

It can be easily checked that for any monomial order, $\ini_<(I_{N_5})$ is squarefree where $N_5$ is  the pentagon lattice,
while $\ini_<(I_{N_3})$ is not squarefree where $N_3$ is the diamond lattice.

\begin{Conjecture}[Squarefree conjecture]
\label{believeitornot}
{\em Let $L$ be a modular lattice. Then for any monomial order $\ini_<(I_L)$ is not squarefree, unless $L$ is distributive.}
\end{Conjecture}

\section{Characterization of planar distributive lattices}

Let $\NN^2$ denote the (infinite) distributive lattice consisting of all pairs $(i, j)$  of nonnegative
integers with the partial order $(i, j)\leq  (k, l)$ if and only if  $i\leq  k$ and $j\leq l$. A {\em planar distributive lattice}
is a finite sublattice $L$ of $\NN^2$   with $(0, 0)\in D$.

\begin{Theorem}
\label{alsoeasy}
Let $L$ be a finite modular  lattice. Then  the following conditions are equivalent:
\begin{enumerate}
\item[{\em (i)}] $L$ is a planar distributive lattice;
\item[{\em (ii)}] $I_L$ has a squarefree  initial  ideal with respect to the lexicographic order for any order of the variables;
\item[{\em (iii)}] $I_L$ has a squarefree  initial  ideal with respect to any monomial order.
\end{enumerate}
\end{Theorem}

\begin{proof}
(i) \implies (iii): Let $m$ and $n$ be the smallest integers such that $L\subset [m]\times [n]$, and consider the polynomial ring $T=K[t_1,\ldots,t_m,s_1,\ldots,s_n]$. We define a $K$-algebra homomorphism $\varphi\: R(L)\to T$ which assigns to $a=(i,j)$ the monomial $s_it_j\in T$. The image $A$ of $\varphi$ is the edge ring of a bipartite graph $G$. The basic relations $ab=(a\vee b)(a\wedge b)$ of $R(L)$ are mapped under $\varphi$ to relations of $A$ corresponding to $4$-cycles of the bipartite graph $G$.
It  is shown in \cite[Theorem 1.2]{OH} that the defining ideal of the edge ring $A$ is generated by the binomials corresponding to 4-cycles, if each even cycle of length $\geq 6$ has a chord. That $G$ satisfies this property  is shown by Querishi  \cite{Q}. It follows that $R(L)\iso A$. Hence we may identify $I_L$ with the toric edge ideal $J$ defining $A$. Next we use a result of Sturmfels (see \cite[Chapter 9]{Stu}) which says the universal Gr\"obner basis of the toric edge ideal of bipartite graph consists of the binomials corresponding to the even cycles with no chords. From this it follows that $\ini_<(J)$ is squarefree for any monomial order.

(iii) \implies (ii) is trivial.

(ii) \implies (i): Suppose $L$ is not planar, then $L$ contains a sublattice which is isomorphic to the Boolean lattice  $B_3$ of rank 3 as shown in Figure~\ref{Fig2}.

\begin{figure}[h]
\begin{center}
\psset{unit=0.7cm}
\begin{pspicture}(1,-1.5)(5,3)

\psline(3,3)(1.5,1.5)
\psline(3,3)(3,1.5)
\psline(3,3)(4.5,1.5)
\psline(1.5,1.5)(1.5,0)
\psline(1.5,1.5)(3,0)
\psline(3,1.5)(1.5,0)
\psline(3,1.5)(4.5,0)
\psline(4.5,1.5)(4.5,0)
\psline(4.5,1.5)(3,0)
\psline(1.5,0)(3,-1.5)
\psline(3,0)(3,-1.5)
\psline(4.5,0)(3,-1.5)

\rput(3,3){$\bullet$}
\put(3,3.2){$a$}
\rput(1.5,1.5){$\bullet$}
\put(1.05,1.4){$b$}
\rput(3,1.5){$\bullet$}
\put(2.6,1.5){$c$}
\rput(1.5,0){$\bullet$}
\put(1.1,-0.1){$e$}
\rput(4.5,1.5){$\bullet$}
\put(4.7,1.4){$d$}
\rput(3,0){$\bullet$}
\put(2.55,-0.25){$f$}
\rput(4.5,0){$\bullet$}
\put(4.7,-0.1){$g$}
\rput(3,-1.5){$\bullet$}
\put(2.8,-2){$h$}

\end{pspicture}
\end{center}
\caption{}\label{Fig3}
\end{figure}

Let $<$ be the lexicographic order induced by an  ordering such that $g<f<e<h<a<d<c<b$ and $b<q$ for any other $q\in L$. The initial ideal of $I_{B_3}$ contains the monomial $ah^2$  in the minimal set of monomial generators. Since $<$ is an elimination order (see \cite[Exercise 2.9]{HH}) it follows that $ah^2$ belongs to the minimal set of monomial generators of  $I_L$.
\end{proof}

In contrast to the order given in the proof of the preceding theorem, there exist lexicographic orders such that  $I_{B_3}$  is quadratic, or not quadratic but squarefree. The question arises whether for any finite distributive lattice there exists a lexicographic monomial order such that $\ini_<(I_L)$ is squarefree.

Recall that the {\em divisor lattice} of a positive integer $n$ is the lattice $D_n$ consisting of all divisors of $n$ ordered by divisibility.  Every divisor lattice is a distributive lattice.

Let, in general, $L$ be a finite pure lattice.  A {\em cut edge} of $L$ is a pair $(a,b)$ of elements of $L$ with $\rank(b) = \rank(a)+1$ such that
$$|\{c\in L\:\; \rank(c)=\rank(a)\}|= |\{c\in L\:\; \rank(c)=\rank(b)\}|=1.$$

\begin{Theorem}
\label{againeasy}
Let $L$ be a finite lattice with no cut edges. Then the following conditions are equivalent:
\begin{enumerate}
\item[{\em (i)}] $L$ is the divisor lattice of $2\cdot 3^r$ for some $r\geq 1$;
\item[{\em (ii)}] $I_L$ has a quadratic Gr\"obner basis with respect to any monomial order.
\end{enumerate}
\end{Theorem}

\begin{proof} (i)\implies (ii): The ideal $I_L$ can be identified with the toric edge ideal of the complete bipartite graph of type $(2,r)$, since $L$ has no cut edges. 
Each cycle in a bipartite graph of type $(2,r)$ is of length $4$. Hence by using again \cite[Chapter 9]{Stu}, the basic binomials of $I_L$ form a universal Gr\"obner basis. This yields the desired conclusion.

(ii)\implies (i): As we have seen in the proofs of Theorem~\ref{easy} and in Theorem~\ref{alsoeasy} that the lattice $L$ cannot contain as a sublattice the diamond lattice and the Boolean lattice $B_3$. It also cannot contain the pentagon lattice $N_5$ of Figure~\ref{Fig1}. Indeed, if we choose the  lexicographic order induced by $c<e<a<d<b$, then $bae$ is a minimal generator of $\ini_<(I_{N_5})$. This implies that $L$ is a planar distributive lattice. Finally, $L$ cannot contain as a sublattice the planar distributive lattice $C_2$ of Figure~\ref{Fig3}.

\begin{figure}[h]
\begin{center}
\psset{unit=0.6cm}
\begin{pspicture}(1,-1.5)(5,3)

\pspolygon(3,3)(4,2)(2,0)(3,-1)(4,0)(2,2)(3,3)

\rput(3,3){$\bullet$}
\put(2.9,3.3){$a$}
\rput(2,2){$\bullet$}
\put(1.5,1.9){$b$}
\rput(3,1){$\bullet$}
\put(3.3,0.9){$d$}
\rput(4,0){$\bullet$}
\put(4.2,-0.1){$f$}
\rput(3,-1){$\bullet$}
\put(2.9,-1.5){$g$}
\rput(2,0){$\bullet$}
\put(1.5,-0.1){$e$}
\rput(4,2){$\bullet$}
\put(4.2,1.9){$c$}

\end{pspicture}
\end{center}
\caption{}\label{Fig2}
\end{figure}

 In fact, if we choose the lexicographic order induced by $g<f<e<c<b<a<d$, then $aef$ is a minimal generator of $\ini_<(I_{N_5})$.  Hence $L$ must be the divisor lattice of $2\cdot 3^r$ for some $r \geq 2$.
\end{proof}

As a strengthening of Theorem~\ref{againeasy} we expect

\begin{Conjecture}[Quadratic conjecture]
{\em Let $I$ be an ideal generated by binomials such that  $I$ has a quadratic Gr\"obner basis with respect to any monomial order. Then either the  generators of $I$ are binomials in pairwise different sets of variables or $I=I_L$ where $L$ is the  divisor lattice of $2\cdot 3^r$ for some $r\geq 1$.}
\end{Conjecture}


\begin{thebibliography}{10}
\bibitem{AHH} A. Aramova, J.\ Herzog, T.\ Hibi,  Finite lattices and lexicographic Gr\"obner bases, {\em Europ. J. Combinatorics} {\bf 21} (2000), 431--439.

\bibitem{HH} J.\ Herzog, T.\ Hibi, ``Monomial Ideals," GTM 260, Springer--Verlag, 2010.

\bibitem{H}
T.\ Hibi, Distributive lattices, affine semigroup rings and
algebras with straightening laws, {\em in} ``Commutative Algebra
and Combinatorics'' (M.\ Nagata and H.\ Matsumura, Eds.), Advanced
Studies in Pure Math., Volume 11, North--Holland, Amsterdam, 1987,
pp. 93 -- 109.

\bibitem{OH} H.\ Ohsugi, T.\ Hibi, Toric ideals generated by quadratic binomials, {\em J.\ Alg.} {\bf 218}  (1999), 509--527.


\bibitem{Q} A.\ Qureshi, Ideals genrated by $2$-minors, collections of cells and stackpolyminoes, in preparation.

\bibitem{St}
R.\ P.\ Stanley, ``Enumerative Combinatorics, Volume I,''
Wadsworth \& Brooks/Cole, Monterey, CA, 1986.

\bibitem{Stu} B.\ Sturmfels, ``Gr\"obner Bases and Convex Polytopes," University Lect.\ Ser., Volume 8, Amer.\ Math.\ Soc., 1996.







\end{thebibliography}
\end{document}